\documentclass[11pt]{amsart}
\usepackage{amssymb,amsmath,amsthm,newlfont,enumitem}

\theoremstyle{plain}
\newtheorem{theorem}{Theorem}[section]
\newtheorem{lemma}[theorem]{Lemma}

\newtheorem{question}[theorem]{Question}

\newcommand{\CC}{{\mathbb C}}
\newcommand{\DD}{{\mathbb D}}
\newcommand{\RR}{{\mathbb R}}
\newcommand{\cU}{{\cal U}}
\renewcommand{\hat}{\widehat}
\DeclareMathOperator{\inter}{int}
\DeclareMathOperator{\psh}{PSH}

\begin{document}

\title{A uniform boundedness principle in pluripotential theory}

\author[Kosi\'nski]{\L ukasz Kosi\'nski}
\address{Institute of Mathematics, Jagiellonian University, \L ojasiewicza 6, 30-348 Krak\'ow, Poland.}
\email{lukasz.kosinski@im.uj.edu.pl}
\thanks{\L K supported by the Ideas Plus grant 0001/ID3/2014/63 of the Polish Ministry of Science
and Higher Education. }

\author[Martel]{\'Etienne Martel}
\address{D\'epartement d'informatique et de g\'enie logiciel, Universit\'e Laval, Qu\'ebec City (Qu\'ebec), Canada G1V 0A6}
\email{etienne.martel.1@ulaval.ca}
\thanks{EM supported by an NSERC undergraduate student research award}

\author[Ransford]{Thomas Ransford}
\address{D\'epartement de math\'ematiques et de statistique, Universit\'e Laval, 
Qu\'ebec City (Qu\'ebec),  Canada G1V 0A6,.}
\email{thomas.ransford@mat.ulaval.ca}
\thanks{TR supported by grants from NSERC and the Canada research chairs program}

\date{7 August 2017}

\begin{abstract}
For families of continuous plurisubharmonic functions
we show that, in a local sense,  separately bounded above implies bounded above.
\end{abstract}

\maketitle

\section{The uniform boundedness principle}\label{S:ubp}

Let $\Omega$ be an open subset of $\CC^N$. 
A function  $u:\Omega\to[-\infty,\infty)$ is called
\emph{plurisubharmonic} if:
\begin{enumerate}
\item $u$ is upper semicontinuous, and
\item $u|_{\Omega\cap L}$ is subharmonic, for each complex line $L$.
\end{enumerate}
For background information on plurisubharmonic functions, 
we refer to the book of Klimek \cite{Kl91}.

It is apparently an open problem whether in fact (2) implies (1) if $N\ge2$. 
In attacking this problem, 
we have repeatedly run up against an
obstruction in the form of a  uniform boundedness principle for plurisubharmonic functions.
This principle, which we think is of interest in its own right, is the main subject of this note.
Here is the formal statement.

\begin{theorem}[Uniform boundedness principle]\label{T:ubp}
Let $D\subset\CC^N$ and $G\subset\CC^M$ be domains, where $N,M\ge1$, and let $\cU$ be a family of continuous 
plurisubharmonic functions on $D\times G$. Suppose that:
\begin{enumerate}[label=\normalfont(\roman*)]
\item $\cU$ is locally uniformly bounded above on $D\times\{w\}$, for each $w\in G$;
\item $\cU$ is locally uniformly bounded above on $\{z\}\times G$, for each $z\in D$.
\end{enumerate}
Then:
\begin{enumerate}[resume, label=\normalfont(\roman*)]
\item $\cU$ is locally uniformly bounded above on $D\times G$.
\end{enumerate}
\end{theorem}

In other words, if there is an upper bound for $\cU$ on each compact subset of $D\times G$ of the form $K\times\{w\}$ or $\{z\}\times L$, then there is an upper bound for $\cU$ on every compact subset of $D\times G$. The point is that  we have no \textit{a priori} quantitative information about these upper bounds, merely that they exist. In this respect, the result resembles the classical Banach--Steinhaus theorem from functional analysis.

The proof of Theorem~\ref{T:ubp} is based on two well-known but non-trivial results from several complex variables: the equivalence (under appropriate assumptions) of plurisubharmonic hulls and polynomial hulls, and Hartogs' theorem on separately holomorphic functions. The details of the proof are presented in \S\ref{S:pfubp}.

The Banach--Steinhaus theorem is usually stated as saying that a family of bounded linear operators on a Banach space $X$ that is pointwise-bounded on $X$ is automatically 
norm-bounded. There is a stronger version of the result in which one assumes merely that the operators are pointwise-bounded on a non-meagre subset $Y$ of $X$, but with the same conclusion. This sharper form leads to new applications (for example, a nice one in the theory of Fourier series can be found in \cite[Theorem~5.12]{Ru87}). Theorem~\ref{T:ubp} too possesses a sharper form, in which one of the conditions (i),(ii) is merely assumed to hold on a non-pluripolar set. This improved version of theorem is the subject of \S\ref{S:ubpgen}.

We conclude the paper in \S\ref{S:appl} by considering applications of these results, and we also discuss the connection with the upper semicontinuity problem mentioned at the beginning of the section.

\section{Proof of the uniform boundedness principle}\label{S:pfubp}

We shall need two auxiliary results. The first one concerns hulls. Given a compact subset $K$ of $\CC^N$, its
\emph{polynomial hull} is defined by
\[
\hat{K}:=\{z\in\CC^N:|p(z)|\le\sup_K|p|\text{~for every polynomial~$p$ on $\CC^N$}\}.
\]
Further, given an open subset $\Omega$ of $\CC^N$ containing $K$, the \emph{plurisubharmonic hull} of $K$ with respect to $\Omega$ is defined by
\[
\hat{K}_{\psh(\Omega)}:=\{z\in\Omega:u(z)\le\sup_Ku\text{~for every plurisubharmonic $u$ on $\Omega$}\}.
\]
Since $|p|$ is plurisubharmonic on $\Omega$ for every polynomial $p$,
it is evident that $\hat{K}_{\psh(\Omega)}\subset \hat{K}$.
In the other direction, we have the following result.

\begin{lemma}[\protect{\cite[Corollary~5.3.5]{Kl91}}]\label{L:hulls}
Let $K$ be a compact subset of $\CC^N$ and let $\Omega$ be an open subset of $\CC^N$ such that $\hat{K}\subset\Omega$. Then
$\hat{K}_{\psh(\Omega)}=\hat{K}$.
\end{lemma}

The second result that we shall need is Hartogs' theorem \cite{Ha06} that separately holomorphic functions are holomorphic.

\begin{lemma}\label{L:Hartogs}
Let $D\subset\CC^N$ and $G\subset\CC^M$ be domains, 
and let $f:D\times G\to\CC$ be a function such that:
\begin{itemize}
\item $z\mapsto f(z,w)$ is holomorphic on $D$, for each $w\in G$;
\item $w\mapsto f(z,w)$ is holomorphic on $G$, for each $z\in D$.
\end{itemize}
Then $f$ is holomorphic on $D\times G$.
\end{lemma}

\begin{proof}[Proof of Theorem~\ref{T:ubp}]
Suppose that the result is false. 
Then there exist sequences $(a_n)$ in $D\times G$ and $(u_n)$
in $\cU$ such that  $u_n(a_n)>n$ for all $n$ and $a_n\to a\in D\times G$. Let $P$ be a compact polydisk with centre $a$ such that $P\subset D\times G$. For each $n$, set
\[
P_n:=\{\zeta\in P:u_n(\zeta)\le n\}.
\]
Then $P_n$ is compact, because the functions in $\cU$ are assumed continuous. Further, since $P$ is convex, we have 
$\hat{P_n}\subset P\subset D\times G$.
By Lemma~\ref{L:hulls}, we have $\hat{P_n}=\hat{(P_n)}_{\psh(D\times G)}$. 
As $a_n$ clearly lies outside this plurisubharmonic hull, 
it follows that $a_n$ also lies outside the polynomial hull of $P_n$. Thus there exists a polynomial $q_n$
such that $\sup_{P_n}|q_n|<1$ and $|q_n(a_n)|>1$. Let $r_n$ be a polynomial vanishing at $a_1,\dots,a_{n-1}$
but not at $a_n$, and set $p_n:=q_n^mr_n$, where $m$ is chosen large enough so that 
\[
\sup_{P_n}|p_n|<2^{-n}
\quad\text{and}\quad
|p_n(a_n)|>n+\sum_{k=1}^{n-1}|p_k(a_n)|.
\]
Let us write $P=Q\times R$, where $Q,R$ are compact polydisks such that $Q\subset D$ and $R\subset G$.
Then, for each $w\in R$, the family  $\cU$ is uniformly bounded above on $Q\times\{w\}$,
so eventually $u_n\le n$ on $Q\times\{w\}$. For these $n$, we then have $Q\times\{w\}\subset P_n$ and hence $|p_n|\le 2^{-n}$ on $Q\times\{w\}$. Thus the series 
\[
f(z,w):=\sum_{n\ge1}p_n(z,w)
\]
converges uniformly on $Q\times\{w\}$. Likewise, it converges uniformly on $\{z\}\times R$, for each $z\in D$. We deduce that:
\begin{itemize}
\item $z\mapsto f(z,w)$ is holomorphic on $\inter(Q)$, for each $w\in \inter(R)$;
\item $w\mapsto f(z,w)$ is holomorphic on $\inter(R)$, for each $z\in \inter(Q)$.
\end{itemize}
By Lemma~\ref{L:Hartogs}, $f$ is holomorphic on $\inter(P)$.
On the other hand, for each $n$, our construction gives
\[
|f(a_n)|
\ge|p_n(a_n)|-\sum_{k=1}^{n-1}|p_k(a_n)|-\sum_{k=n+1}^\infty| p_k(a_n)|>n.
\]
Since $a_n\to a$, it follows that $f$ is discontinuous at $a$, the central point of $P$. We thus have arrived at a contradiction, and the proof is complete.
\end{proof}

One might wonder if Theorem~\ref{T:ubp} remains true if we drop one of the assumptions (i) or (ii). Here is a simple example to show that it does not. For each $n\ge1$, set
\[
K_n:=\{z\in\CC:|z|\le n,~1/n\le \arg(z)\le 2\pi\},
\]
and let $(z_n)$ be a sequence such that $z_n\in\CC\setminus K_n$ for all $n$ and $z_n\to0$. By Runge's theorem, for each $n$ there exists a polynomial $p_n$ such that $\sup_{K_n}|p_n|\le 1$  and $|p_n(z_n)|>n$. The sequence $|p_n|$ is then pointwise bounded on $\CC$, but not uniformly bounded in any neighborhood of $0$. 
Thus, if we define $u_n(z,w):=|p_n(z)|$,
then we obtain a sequence of continuous plurisubharmonic functions on $\CC\times\CC$ satisfying (ii) but not (iii).

Although we cannot drop (i) or (ii) altogether, it \emph{is} possible to weaken one of the conditions (i) or (ii)  to hold merely on a set that is `not too small', and still obtain the conclusion (iii). This is the subject of next section.

\section{A stronger form of the uniform boundedness principle}\label{S:ubpgen}

A subset $E$ of $\CC^N$ is called \emph{pluripolar} if there exists a plurisubharmonic function $u$ on $\CC^N$ such that $u=-\infty$ on $E$ but $u\not\equiv-\infty$ on $\CC^N$. Pluripolar sets have Lebesgue measure zero, and a countable union of pluripolar sets is again pluripolar. For further background on pluripolar sets, we again refer to Klimek's book \cite{Kl91}.

In this section we establish the following generalization of Theorem~\ref{T:ubp}, in which we weaken one of the assumptions (i),(ii) to hold merely on a non-pluripolar set.

\begin{theorem}\label{T:ubpgen}
Let $D\subset\CC^N$ and $G\subset\CC^M$ be domains, where $N,M\ge1$, 
and let $\cU$ be a family of continuous 
plurisubharmonic functions on $D\times G$. Suppose that:
\begin{enumerate}[label=\normalfont(\roman*)]
\item $\cU$ is locally uniformly bounded above on $D\times\{w\}$, for each $w\in G$;
\item $\cU$ is locally uniformly bounded above on $\{z\}\times G$, for each $z\in F$,
\end{enumerate}
where $F$ is a non-pluripolar subset of $D$. Then:
\begin{enumerate}[resume, label=\normalfont(\roman*)]
\item $\cU$ is locally uniformly bounded above on $D\times G$.
\end{enumerate}
\end{theorem}

For the proof, we need the following generalization of Hartogs' theorem, due to Terada \cite{Te67,Te72}.

\begin{lemma}\label{L:Terada}
Let $D\subset\CC^N$ and $G\subset\CC^M$ be domains, 
and let $f:D\times G\to\CC$ be a function such that:
\begin{itemize}
\item $z\mapsto f(z,w)$ is holomorphic on $D$, for each $w\in G$;
\item $w\mapsto f(z,w)$ is holomorphic on $G$, for each $z\in F$,
\end{itemize}
where $F$ is a non-pluripolar subset of $D$.
Then $f$ is holomorphic on $D\times G$.
\end{lemma}

\begin{proof}[Proof of Theorem~\ref{T:ubpgen}]
We define two subsets $A,B$ of $D$ as follows. First, $z\in A$ if $w\mapsto\sup_{u\in\cU}u(z,w)$ is locally bounded above on $G$. Second, $z\in B$ if there exists a neighborhood $V$ of $z$ in $D$ such that $(z,w)\mapsto\sup_{u\in\cU}u(z,w)$ is locally bounded above on $V\times G$. Clearly $B$ is open in $D$ and $B\subset A$. Also $F\subset A$, so $A$ is non-pluripolar.

Let $z_0\in D\setminus B$. Then there exists $w_0\in G$ such that $\cU$ is not uniformly bounded above on any neighborhood of $(z_0,w_0)$. The same argument as in the proof of Theorem~\ref{T:ubp} leads to the existence of a compact polydisk $P=Q\times R$ around $(z_0,w_0)$ and a function $f:Q\times R\to\CC$ such that:
\begin{itemize}
\item $z\mapsto f(z,w)$ is holomorphic on $\inter(Q)$, for each $w\in\inter(R)$,
\item $w\mapsto f(z,w)$ is holomorphic on $\inter(R)$, for each $z\in \inter(Q)\cap A$,
\end{itemize}
and at the same time $f$ is unbounded in each neighborhood of $(z_0,w_0)$. By Lemma~\ref{L:Terada}, this is possible only if $\inter(Q)\cap A$ is pluripolar. 

Resuming what we have proved: if $z\in D$ and every neighborhood of $z$ meets $A$ in a non-pluripolar set,
then $z\in B$.

We  now conclude the proof with a connectedness argument. As $A$ is non-pluripolar, and a countable union of pluripolar sets is pluripolar, there exists $z_1\in D$ such that every neighborhood of $z_1$ meets $A$ in a non-pluripolar set, and consequently $z_1\in B$. Thus $B\ne\emptyset$. We have already remarked that $B$ is open in $D$. Finally, if $z\in D\setminus B$, then there is a an open neighborhood $W$ of $z$ that meets $A$ in a pluripolar set, hence $B\cap W$ is both pluripolar and open, and consequently empty. This shows that $D\setminus B$ is open in $D$. As $D$ is connected, we conclude that $B=D$, which proves the theorem.
\end{proof}

We end the section with some remarks concerning the sharpness of Theorem~\ref{T:ubpgen}.

Firstly, we cannot weaken both conditions (i) and (ii) simultaneously. 
Indeed, let $\DD$ be the unit disk, and define a sequence $u_n:\DD\times\DD\to\RR$ by
\[
u_n(z,w):=n(|z+w|-3/2).
\]
Then
\begin{itemize}
\item $z\mapsto\sup_n u_n(z,w)$ is bounded above on $\DD$ for all $|w|\le1/2$,
\item $w\mapsto\sup_n u_n(z,w)$ is bounded above on $\DD$ for all $|z|\le1/2$,
\end{itemize}
but the sequence $u_n(z,w)$ is not even pointwise bounded above at the point $(z,w):=(\frac{4}{5},\frac{4}{5})$.

Secondly, the condition in Theorem~\ref{T:ubpgen} that $F$ be non-pluripolar is sharp, at least for $F_\sigma$-sets. Indeed, let $F$ be an $F_\sigma$-pluripolar subset of $D$. Then there exists a plurisubharmonic function $v$ on $\CC^N$ such that $v=-\infty$ on $F$ and $v(z_0)>-\infty$ for some $z_0\in D$. By convolving $v$ with suitable smoothing functions, we can construct a sequence $(v_n)$ of continuous plurisubharmonic functions on $\CC^N$ decreasing to $v$ and such that the sets $\{v_n\le -n\}$ cover $F$. Let $(p_n)$ be a sequence of polynomials in one variable that is pointwise bounded in $\CC$ but not uniformly bounded on any neighborhood of $0$ (such a sequence was constructed at the end of \S\ref{S:pfubp}). Choose positive integers $N_n$ such that
$\sup_{|w|\le n}|p_n(w)|\le N_n$, and define $u_n:D\times\CC\to\RR$ by
\[
u_n(z,w):=v_{N_n}(z)+|p_n(w)|.
\]
Then 
\begin{itemize}
\item $z\mapsto\sup_n u_n(z,w)$ is locally bounded above on $D$ for all $w\in\CC$,
\item $w\mapsto\sup_n u_n(z,w)$ is locally bounded above on $\CC$ for all $z\in F$,
\end{itemize}
but  $\sup_nu_n(z,w)$ is not bounded above on any neighborhood of $(z_0,0)$.

\section{Applications of the uniform boundedness principle}\label{S:appl}

Our first application is to null sequences of plurisubharmonic functions.

\begin{theorem}\label{T:null}
Let $D\subset\CC^N$ and $G\subset\CC^M$ be domains, 
and let $(u_n)$ be a sequence of positive continuous 
plurisubharmonic functions on $D\times G$. Suppose that:
\begin{itemize}
\item $u_n(\cdot,w)\to0$ locally uniformly on $D$ as $n\to\infty$, for each $w\in G$,
\item $u_n(z,\cdot)\to0$ locally uniformly on $G$ as $n\to\infty$, for each $z\in F$,
\end{itemize}
where $F\subset D$ is non-pluripolar.
Then $u_n\to0$ locally uniformly on $D\times G$.
\end{theorem}

\begin{proof}
Let $a\in D\times G$. Choose $r>0$ such that $\overline{B}(a,2r)\subset D\times G$. 
Writing $m$ for Lebesgue measure on $\CC^N\times \CC^M$,
we have
\begin{align*}
\sup_{\zeta\in\overline{B}(a,r)}u_n(\zeta)
&\le \sup_{\zeta\in\overline{B}(a,r)}\frac{1}{m({B}(\zeta,r))}\int_{{B}(\zeta,r)}u_n\,dm\\
&\le \frac{1}{m({B}(0,r))}\int_{{B}(a,2r)}u_n\,dm.
\end{align*}
Clearly $u_n\to0$ pointwise on $B(a,2r)$. Also, by Theorem~\ref{T:ubpgen}, the sequence $(u_n)$ is uniformly
bounded on $B(a,2r)$. By the dominated convergence theorem, it follows that $\int_{B(a,2r)}u_n\,dm\to0$ as $n\to\infty$. Hence $\sup_{\zeta\in\overline{B}(a,r)}u_n(\zeta)\to0$ as $n\to\infty$.
\end{proof}

Our second application relates to the  problem mentioned at the beginning of \S\ref{S:ubp}. Recall that $u:\Omega\to[-\infty,\infty)$ is plurisubharmonic if
\begin{enumerate}
\item $u$ is upper semicontinuous, and
\item $u|_{\Omega\cap L}$ is subharmonic, for each complex line $L$,
\end{enumerate}
and the problem is to determine whether in fact (2) implies (1).
Here are some known partial results:
\begin{itemize}
\item[-] Lelong \cite{Le45} showed that (2) implies (1) if, in addition,  $u$ is locally bounded above.
\item[-] Arsove \cite{Ar66} generalized Lelong's result by showing that,  if $u$ if separately subharmonic and locally bounded above, then $u$ is upper semicontinuous. (Separately subharmonic means that (2) holds just with lines $L$ parallel to the coordinate axes.) Further results along these lines were obtained in \cite{AG93,KT96,Ri14}.
\item[-] Wiegerinck \cite{Wi88} gave an example of a separately subharmonic function that is not upper semicontinuous. Thus Arsove's result no longer holds without the assumption that $u$ be locally bounded above.
\end{itemize}

In seeking an example to show that (2) does not imply (1), it is natural to try to emulate Wiegerinck's example,
which was constructed  as follows. 
Let $K_n, z_n$ and $p_n$ be defined as in the example at the end of \S\ref{S:pfubp}. For each $n$ define $v_n(z):=\max\{|p_n(z)|-1,\,0\}$. Then $v_n$ is a subharmonic function,  $v_n=0$ on $K_n$ and  $v_n(z_n)>n-1$. Set
\[
u(z,w):=\sum_kv_k(z)v_k(w).
\]
If $w\in\CC$, then $w\in K_n$ for all large enough $n$, so $v_n(w)=0$. Thus, for each fixed $w\in\CC$, the function $z\mapsto u(z,w)$ is a finite sum of subharmonic functions, hence subharmonic. Evidently, the same is true with roles of $z$ and $w$ reversed. Thus $u$ is separately subharmonic. On the other hand, for each $n$ we have
\[
u(z_n,z_n)\ge v_n(z_n)v_n(z_n)>(n-1)^2,
\]
so $u$ is not bounded above on any neighborhood of $(0,0)$.

This example does not answer the question of whether (2) implies (1) because the summands $v_k(z)v_k(w)$ are not plurisubharmonic as functions of $(z,w)\in\CC^2$.
It is tempting to try to modify the construction by replacing $v_k(z)v_k(w)$ by a positive plurisubharmonic sequence $v_k(z,w)$ such that the partial sums $\sum_{k=1}^nv_k$
are locally bounded above on each complex line, but not on any open neighborhood of $(0,0)$.
However, Theorem~\ref{T:ubp} demonstrates immediately  that this endeavor is doomed to failure, at least if we restrict ourselves to continuous plurisubharmonic functions. 

This raises the following question, which, up till now, we have been unable to answer.

\begin{question}\label{Q:cts}
Does Theorem~\ref{T:ubp} remain true without the assumption that the functions in $\cU$ be continuous?
\end{question}

This is of interest because of the following result.

\begin{theorem}
Assume that the answer to Question~\ref{Q:cts} is positive.
Let $\Omega$ be an open subset of $\CC^N$ and let $u:\Omega\to[-\infty,\infty)$ be a function such that $u|_{\Omega\cap L}$ is subharmonic for each complex line $L$. Define
\[ 
s(z):=\sup\{v(z):v \text{~plurisubharmonic on~}\Omega,~v\le u\}.
\]
Then $s$ is plurisubharmonic on $\Omega$. 
\end{theorem}

\begin{proof}
Let $\cU$ be the family of plurisubharmonic functions $v$ on $\Omega$ such that $v\le u$. If the answer to Question~\ref{Q:cts} is positive, then $\cU$ is locally uniformly bounded above on $\Omega$. Hence, by \cite[Theorem~2.9.14]{Kl91}, the upper semicontinuous regularization $s^*$ of $s$ is plurisubharmonic on $\Omega$,
and, by \cite[Proposition~2.6.2]{Kl91},  $s^*=s$ Lebesgue-almost everywhere on $\Omega$. Fix $z\in\Omega$. Then there exists a complex line $L$ passing through $z$ such that $s^*=s$ almost everywhere on $\Omega\cap L$. Let $\mu_r$ be normalized Lebesgue measure on $B(z,r)\cap L$. Then
\[
s^*(z)\le\int_{B(z,r)\cap L} s^*\,d\mu_r=\int_{B(z,r)\cap L} s\,d\mu_r\le\int_{B(z,r)\cap L} u\,d\mu_r.
\]
(Note that $u$ is Borel-measurable by \cite[Lemma~1]{Ar66}.)
Since $u|_{\Omega\cap L}$ is upper semicontinuous, we can let $r\to0^+$ to deduce that $s^*(z)\le u(z)$. Thus $s^*$ is itself a member of $\cU$, so $s^*\le s$, and thus finally $s=s^*$ is plurisubharmonic on $\Omega$.
\end{proof}

Of course, $s=u$ if and only if $u$ is itself plurisubharmonic. 
Maybe this could provide a way of attacking the problem of showing that $u$ is plurisubharmonic?

\end{document}